\documentclass[11pt]{amsart}

\usepackage{amsmath}
\usepackage{amssymb}
\usepackage{amscd}
\usepackage{mathrsfs}

\usepackage{xypic}
\xyoption{all}

\newcommand{\nc}{\newcommand}

\nc{\lan}{\big\langle}
\nc{\ran}{\big\rangle}

\nc{\kk}{{\mathsf{k}}}

\nc{\LL}{{\mathbb{L}}}
\nc{\PP}{{\mathbb{P}}}
\nc{\QQ}{{\mathbb{Q}}}
\nc{\RR}{{\mathbb{R}}}
\nc{\TT}{{\mathbb{T}}}
\nc{\ZZ}{{\mathbb{Z}}}

\nc{\CA}{{\mathcal{A}}}
\nc{\CB}{{\mathcal{B}}}
\nc{\D}{{\mathcal{D}}}
\nc{\CE}{{\mathcal{E}}}
\nc{\CK}{{\mathcal{K}}}
\nc{\CL}{{\mathcal{L}}}
\nc{\CM}{{\mathcal{M}}}
\nc{\CN}{{\mathcal{N}}}
\nc{\CO}{{\mathcal{O}}}
\nc{\CS}{{\mathcal{S}}}
\nc{\CT}{{\mathcal{T}}}
\nc{\CU}{{\mathcal{U}}}
\nc{\CX}{{\mathcal{X}}}

\nc{\BB}{{\mathbf{B}}}
\nc{\BD}{{\mathbf{D}}}
\nc{\BG}{{\mathbf{G}}}
\nc{\BO}{{\mathbf{O}}}
\nc{\BP}{{\mathbf{P}}}
\nc{\BT}{{\mathbf{T}}}
\nc{\BU}{{\mathbf{U}}}
\nc{\bm}{{\mathbf{m}}}
\nc{\bu}{{\mathbf{u}}}
\nc{\bx}{{\mathbf{x}}}

\nc{\SJ}{{\mathsf{J}}}

\nc{\TBP}{{\tilde{\BP}}}

\nc{\TD}{{\widetilde{\D}}}
\nc{\TCA}{{\tilde{\CA}}}
\nc{\TY}{{\widetilde{Y}}}

\nc{\RHom}{\mathop{\mathsf{RHom}}\nolimits}
\nc{\Hom}{\mathop{\mathsf{Hom}}\nolimits}
\nc{\Ext}{\mathop{\mathsf{Ext}}\nolimits}
\nc{\RCHom}{\mathop{\mathbf{R}\mathcal{H}\mathit{om}}\nolimits}
\nc{\CHom}{\mathop{\mathcal{H}\mathit{om}}\nolimits}
\nc{\CExt}{\mathop{\mathcal{E}\mathit{xt}}\nolimits}
\nc{\Tor}{\mathop{\mathsf{Tor}}\nolimits}
\nc{\Pic}{\mathop{\mathsf{Pic}}\nolimits}

\nc{\Br}{\mathop{\mathsf{Br}}\nolimits}
\nc{\Cone}{\mathop{\mathsf{Cone}}\nolimits}
\nc{\codim}{\mathop{\mathsf{codim}}\nolimits}
\nc{\sing}{{\mathsf{sing}}}
\nc{\perf}{{\mathsf{perf}}}
\nc{\Ker}{{\mathsf{Ker}}}
\nc{\End}{{\mathsf{End}}}
\nc{\Tr}{{\mathsf{Tr}}}
\nc{\Pf}{{\mathsf{Pf}}}
\nc{\Gr}{{\mathsf{Gr}}}
\nc{\SGr}{{\mathsf{SGr}}}
\nc{\LGr}{{\mathsf{LieGr}}}
\nc{\GTGr}{{\BG_2\mathsf{Gr}}}
\nc{\OGr}{{\mathsf{OGr}}}
\nc{\OFl}{{\mathsf{OFl}}}
\nc{\Fl}{{\mathsf{Fl}}}
\nc{\Spin}{{\mathsf{Spin}}}
\nc{\GL}{{\mathsf{GL}}}
\nc{\SL}{{\mathsf{SL}}}
\nc{\fg}{{\mathfrak{g}}}
\nc{\fsl}{{\mathfrak{sl}}}
\nc{\fso}{{\mathfrak{so}}}
\nc{\PGL}{{\mathsf{PGL}}}
\nc{\ch}{{\mathsf{ch}}}
\nc{\td}{{\mathsf{td}}}
\nc{\id}{{\mathsf{id}}}

\theoremstyle{plain}

\newtheorem{theorem}{Theorem}[section]

\newtheorem{lemma}[theorem]{Lemma}
\newtheorem{proposition}[theorem]{Proposition}
\newtheorem{corollary}[theorem]{Corollary}

\theoremstyle{definition}

\newtheorem{definition}[theorem]{Definition}

\theoremstyle{remark}

\title[A simple counterexample to the Jordan--H\"older property for derived categories]{A simple counterexample to the Jordan--H\"older property\\for derived categories}
\author{Alexander Kuznetsov}
\address{\sloppy
\parbox{0.9\textwidth}{
Algebra Section, Steklov Mathematical Institute,
8 Gubkin str., Moscow 119991 Russia
\hfill\\[5pt]
The Poncelet Laboratory, Independent University of Moscow
\hfill\\[5pt]
Laboratory of Algebraic Geometry, SU-HSE.
\hfill
}\bigskip}
\email{akuznet@mi.ras.ru}
\date{}
\thanks{I was partially supported by
RFFI grants 11-01-00393, 11-01-00568, 12-01-33024, NSh-5139.2012.1,
the grant of the Simons foundation, and by AG Laboratory SU-HSE, RF government grant, ag.11.G34.31.0023.}

\begin{document}

\begin{abstract}
A counterexample to the Jordan--H\"older property for semiorthogonal decompositions
of derived categories of smooth projective varieties was constructed by B\"ohning, Graf von Bothmer and Sosna.
In this short note we present a simpler example by realizing Bondal's quiver
in the derived category of a blowup of $\PP^3$.
\end{abstract}

\maketitle

\section{Introduction}

Given a triangulated category $\CT$, a {\sf semiorthogonal decomposition for $\CT$} is a chain
$\CA_1,\dots,\CA_m \subset \CT$ of full triangulated subcategories such that 
\begin{itemize}
\item for any $j > i$ one has $\Hom(\CA_j,\CA_i) = 0$, and
\item for any object $T \in \CT$ there is a chain of morphisms
\begin{equation*}
0 = T_m \to T_{m-1} \to \dots \to T_1 \to T_0 = T
\end{equation*}
such that $\Cone(T_i \to T_{i-1}) \in \CA_i$.
\end{itemize}
We write $\CT = \langle \CA_1,\dots,\CA_m \rangle$ to denote a semiorthogonal decomposition.

It is well known \cite{BK} that the braid group acts on the set of all semiorthogonal decompositions of a given category ---
the $i$-th generator of the braid group acts as
\begin{equation*}
\langle \CA_1,\dots,\CA_r \rangle \mapsto
\langle \CA_1,\dots,\CA_{i-1},\CA_{i+1},{}^\perp \langle \CA_1,\dots,\CA_{i-1},\CA_{i+1} \rangle \cap \langle \CA_{i+2},\dots,\CA_m\rangle^\perp,\CA_{i+2},\dots,\CA_m \rangle.
\end{equation*}
So if a category $\CT$ has a semiorthogonal decomposition, it has many of them. However, 
it is also well known that the equivalence classes of the components do not change under this action ---
there is an equivalence of categories
\begin{equation*}
{}^\perp \langle \CA_1,\dots,\CA_{i-1},\CA_{i+1} \rangle \cap \langle \CA_{i+2},\dots,\CA_r\rangle^\perp \cong \CA_i
\end{equation*}
This motivates the following definition.

\begin{definition}
A triangulated category $\CT$ {\sf has the Jordan--H\"older property} if for any pair
\begin{equation*}
\CT = \langle \CA_1,\dots,\CA_m \rangle,
\qquad
\CT = \langle \CB_1,\dots,\CB_n \rangle
\end{equation*}
of semiorthogonal decompositions with indecomposable components one has $m = n$ and there is a permutation $\sigma \in S_m$
such that $\CB_i \cong \CA_{\sigma(i)}$ for each $1 \le i \le m$.
\end{definition}

Among triangulated categories of geometrical nature there are very few for which the Jordan-H\"older property has been proved.
Basically these consist of those for which all semiorthogonal decompositions can be classified
(basically these are $\BD(\PP^1)$ and its quotient stacks $\BD(\PP^1/\Gamma)$, see~\cite{Kir}
for the detailed investigation of the latter), or those
which are themselves indecomposable (connected Calabi--Yau categories~\cite{Bri}, derived categories 
of curves of positive genus~\cite{Oka}). Even for $\PP^2$ the property is questionable. 

On one hand, if the Jordan--H\"older property could be justified for derived categories
of smooth projective varieties, this would allow to define nice birational invariants,
see~\cite{KECM} and~\cite{BBS}.
On the other hand, it was known for a long time that the property is not satisfied for 
arbitrary triangulated categories. A very simple counterexample was constructed by Alexei Bondal
long ago. Namely, Bondal considered a quiver with relations
\begin{equation}\label{bq}
Q = \left(\left. \xymatrix@C=3em{ 
\bullet \ar@<1ex>[r]^{\alpha_1} \ar@<-1ex>[r]_{\alpha_2} & 
\bullet \ar@<1ex>[r]^{\beta_1} \ar@<-1ex>[r]_{\beta_2} & 
\bullet 
} \right|\
\beta_1\alpha_2 = \beta_2\alpha_1 = 0
\right)
\end{equation}
and noted that on one hand as any oriented quiver it has a full exceptional collection
\begin{equation*}
\BD(Q) = \langle P_1,P_2,P_3 \rangle
\end{equation*}
with $P_i$ being the projective module of the $i$-th vertex, and on the other hand,
it has an exceptional object
\begin{equation}\label{oe}
P = \left( \xymatrix@C=3em{ 
\kk \ar@<1ex>[r]^{1} \ar@<-1ex>[r]_{0} & 
\kk \ar@<1ex>[r]^{1} \ar@<-1ex>[r]_{0} & 
\kk
} \right)
\end{equation}
which is {\sf nonextendable}, i.e.\ does not extend to a longer exceptional collection (even numerically).

Since the category $\BD(Q)$ itself is not equivalent to the derived category of a scheme, Bondal's
counterexample does not answer the question whether the Jordan--H\"older property is true for derived 
categories of schemes, so for some time there was a little hope that by some miracle it might be true.
A recent paper of B\"ohning, Graf von Bothmer and Sosna~\cite{BBS} gave finally a negative answer
to this question. To be more precise, investigating the derived category $\BD(X)$ of the classical 
Godeaux surface 
\begin{equation*}
X = \{ x_1^5 + x_2^5 + x_3^5 + x_4^5 = 0 \}/\ZZ_5 \subset \PP^3/\ZZ_5,
\end{equation*}
where $\ZZ_5$ acts with weight $i$ on $x_i$, the authors constructed two nonextendable exceptional
collections in~$\BD(X)$, one of length 11 (the maximal possible), and the other of length 9.
The nonextendability is also checked on the numerical level. The construction of this
counterexample and the proofs are rather complicated up to such extent that at some moment
a computer computation (using Macaulay2) is used. 

The goal of this note is to give an elementary example. We note that although $\BD(Q)$ itself is not
equivalent to the derived category of an algebraic variety, it can be realized as a semiorthogonal
component of such. And this is enough to get a counterexample. The variety we consider is a two-step
blowup of $\PP^3$ in two smooth rational curves. Its derived category has a full exceptional collection
and we observe that it contains Bondal's quiver as a subcollection.

%
%
%

\section{Example}

Let $A \subset V$ be a pair of vector spaces of dimensions 2 and 4 respectively, so that $\PP(A) \subset \PP(V)$
is a line $\PP^1$ in a $\PP^3$. Let $X$ be the blowup of $\PP(V)$ along $\PP(A)$ with $E$ being the exceptional divisor.
Then
\begin{equation*}
E \cong \PP(A)\times\PP(V/A) = \PP^1\times\PP^1.
\end{equation*}
We denote by $i$ the embedding of $E$ into $X$ and by $H$ the pullback to $X$ of a hyperplane on $\PP(V)$.
The Picard group of $X$ is generated by $H$ and $E$ and we have
\begin{equation}\label{oxhe}
\CO_X(H)_{E} \cong \CO_E(1,0).
\end{equation}
%
%
%
%

Let $C$ be a smooth rational curve on $X$ which intersects $E$ transversally in two points
\begin{equation*}
P_1 = (a_1,b_1),\qquad
P_2 = (a_2,b_2),
\end{equation*}
where $a_i \in \PP(A)$, $b_i \in \PP(V/A)$ and with
\begin{equation*}
a_1\ne a_2.
\end{equation*}
For example, one can take $C$ to be the proper preimage of a conic in $\PP(V)$ intersecting the line $\PP(A)$ in two distinct points
(in this case the points $b_1$ and $b_2$ will coincide, but this does not matter).

Let $\pi:Y \to X$ be the blowup of $X$ in $C$. Let $E'$ be the exceptional divisor of this blowup,
$i':E' \to Y$ be its embedding into $Y$, $p:E' \to C$ the projection, and $j:C \to X$ the embedding of the curve.
This can be summarized in a diagram
\begin{equation*}
\xymatrix{
E' \ar[r]^{i'} \ar[d]_p & 
Y \ar[d]^\pi \\
C \ar[r]^{j} &
X
}
\end{equation*}

Recall that by Orlov's blowup formula~\cite{Or} we have a semiorthogonal decomposition
\begin{equation*}
\BD(Y) = \langle \pi^*(\BD(X)), i'_*p^*(\BD(C)) \rangle.
\end{equation*}
We take the following triple of sheaves on $Y$:
\begin{equation*}
\CE_1 = \CO_Y = \pi^*\CO_X,\qquad
\CE_2 = \pi^*i_*\CO_E(1,0),\qquad
\CE_3 = i'_*p^*\CO_C(-3).
\end{equation*}

\begin{lemma}
The triple $(\CE_1,\CE_2,\CE_3)$ is exceptional and extends to a full exceptional collection in $\BD(Y)$.
\end{lemma}
\begin{proof}
First, we extend the pair $(\CO_X,i_*\CO_E(1,0))$ to a full exceptional collection in $\BD(X)$:
\begin{equation*}
\BD(X) = \langle \CO_X(-3H),\CO_X(-2H),\CO_X(-H),\CO_X,i_*\CO_E,i_*\CO_E(1,0) \rangle.
\end{equation*}
This is just the full exceptional collection obtained by combining the standard exceptional collection 
$(\CO(-3),\CO(-2),\CO(-1),\CO)$ on $\PP(V)$ with the standard collection $(\CO,\CO(1))$ on $\PP(A)$
if we consider $X$ as the blowup of $\PP(V)$ in $\PP(A)$. Pulling it back to $Y$ and combining 
with the exceptional collection $\CO_C(-3),\CO_C(-2)$ on $C$ we obtain a full exceptional collection
in $\BD(Y)$:
\begin{equation*}
\BD(Y) = \langle \CO_Y(-3H),\CO_Y(-2H),\CO_Y(-H),
\underline{\CO_Y},\pi^*i_*\CO_E,
\underline{\pi^*i_*\CO_E(1,0)},\underline{i'_*p^*\CO_C(-3)},i'_*\CO_C(-2) \rangle
\end{equation*}
(we denote here the pullback of $H$ to $Y$ also by $H$). The underlined terms of this exceptional
collection are the objects $\CE_1$, $\CE_2$ and $\CE_3$.
\end{proof}


\begin{lemma}
We have
\begin{equation*}
\Ext^\bullet(\CE_1,\CE_2) = A^*,
\qquad
\Ext^\bullet(\CE_2,\CE_3) = \kk^2[-1],
\qquad
\Ext^\bullet(\CE_1,\CE_3) = \kk^2[-1].
\end{equation*}
\end{lemma}
Here the brackets stand for the homological shift. In other words, it is claimed that between $\CE_1$ and $\CE_2$ there is only $\Hom$,
while from $\CE_1$ and $\CE_2$ to $\CE_3$ there are only $\Ext^1$.

\begin{proof}
Since $\pi^*$ is fully faithful we have
\begin{equation*}
\Ext^\bullet(\CE_1,\CE_2) = \Ext^\bullet(\CO_X,i_*\CO_E(1,0)) = H^\bullet(E,\CO_E(1,0)) = A^*.
\end{equation*}
Furthermore, for any $F \in \BD(X)$, $G \in \BD(C)$ we have
\begin{equation*}
\Ext^\bullet(\pi^*F,i'_*p^*G) \cong
\Ext^\bullet(F,\pi_*i'_*p^*G) \cong
\Ext^\bullet(F,j_*p_*p^*G) \cong
\Ext^\bullet(F,j_*G).
\end{equation*}
It follows that 
\begin{equation*}
\Ext^\bullet(\CE_1,\CE_3) = 
\Ext^\bullet(\CO_X,j_*\CO_C(-3)) = 
H^\bullet(C,\CO_C(-3)) = \kk^2[-1].
\end{equation*}
Finally, 
\begin{equation*}
\Ext^\bullet(\CE_2,\CE_3) = 
\Ext^\bullet(i_*\CO_E(1,0),j_*\CO_C(-3)).
\end{equation*}
To compute the latter we take the resolution
\begin{equation*}
0 \to \CO_X(H-E) \xrightarrow{\ E\ } \CO_X(H) \to i_*\CO_E(1,0) \to 0
\end{equation*}
and apply the local $\CHom(-,j_*\CO_C(-3))$ functor. We deduce that $\RCHom(i_*\CO_E(1,0),j_*\CO_C(-3))$
is quasiisomorphic to the complex
\begin{equation*}
j_*\CO_C(-H-3p) \xrightarrow{\ E\ } j_*\CO_C(E-H-3p), 
\end{equation*}
where $p$ stands for the class of a point on $C$, with terms in grading $0$ and $1$ respectively.
Since $C$ is a smooth curve and $E$ is a section of a line bundle on $C$
vanishing with multiplicity 1 at points $P_1$ and $P_2$ only, we see that
the above map has trivial kernel and its cokernel is just the sum of $\CO_{P_1}$ and $\CO_{P_2}$,
the structure sheaves of the points. It follows that
\begin{equation*}
\CExt^\bullet(i_*\CO_E(1,0),j_*\CO_C(-3)) \cong \CO_{P_1}[-1] \oplus \CO_{P_2}[-1].
\end{equation*}
Using the local-to-global spectral sequence we deduce that
\begin{equation}\label{e23}
\Ext^\bullet(i_*\CO_E(1,0),j_*\CO_C(-3)) = H^\bullet(Y,\CO_{P_1})[-1] \oplus H^\bullet(Y,\CO_{P_2})[-1].
\end{equation}
which gives the last claim of the Lemma.
\end{proof}

It remains to compute the multiplication map. Let $\alpha_1,\alpha_2$ be the basis of $A^*$ dual to the basis
$a_1,a_2$ of $A$ given by the first coordinates of the points $P_1$ and $P_2$. Let $\beta_1,\beta_2$
be the basis of $\Ext^1(\CE_2,\CE_3)$ given by the spaces $H^0(Y,\CO_{P_1})$ and $H^0(Y,\CO_{P_2})$ in~\eqref{e23} respectively.

\begin{proposition}
The multiplication map
\begin{equation*}
m:\Hom(\CE_1,\CE_2) \otimes \Ext^1(\CE_2,\CE_3) \to \Ext^1(E_1,E_3)
\end{equation*}
is surjective and its kernel is spanned by $\alpha_1\otimes\beta_2$ and $\alpha_2\otimes\beta_1$.
\end{proposition}

Before giving a proof let us mention the consequences.

\begin{corollary}
The algebra of homomorphisms of the exceptional collection $\CE_1,\CE_2[1],\CE_3[1]$ is isomorphic
to the path algebra of Bondal's quiver~\eqref{bq}.
\end{corollary}

\begin{corollary}
The derived category $\BD(Y)$ of $Y$ does not have the Jordan--H\"older property.
\end{corollary}

\begin{proof}[Proof of the Proposition]
First, the pullback-pushforward adjunction for the morphism $\pi$ shows that the map $m$ coincides
with the multiplication map
%
\begin{equation*}
\Hom(\CO_X,i_*\CO_E(1,0)) \otimes \Ext^1(i_*\CO_E(1,0),j_*\CO_C(-3)) \to \Ext^1(\CO_X,j_*\CO_C(-3)).
\end{equation*}
Now let us identify the bases in the spaces we are interested in.

The first space $\Hom(\CO_X,i_*\CO_E(1,0))$ has $\alpha_1,\alpha_2$ as a base. By~\eqref{oxhe} 
we can find sections $\bar\alpha_1,\bar\alpha_2$ of $\CO_X(H)$ which restrict to the sections $\alpha_1$ and $\alpha_2$ of $\CO_E(1,0)$. 
These are given just by two planes in $\PP^3 = \PP(V)$ intersecting the line $\PP^1=\PP(A)$ transversally in points 
$a_2$ and $a_1$ respectively. 

Further, denote by $\rho_i$ the canonical morphisms
\begin{equation*}
\CO_E(1,0) \xrightarrow{\ \rho_i\ } \CO_{P_i}
\end{equation*}
and by $\eta_i$ the canonical extensions
\begin{equation*}
\CO_{P_i} \xrightarrow{\ \eta_i\ } \CO_C(-3)[1].
\end{equation*}
Then
\begin{equation*}
\beta_i = j_*(\eta_i) \circ \rho_i.
\end{equation*}
To see this consider (the pushforward to $X$ of) the exact sequence
\begin{equation}\label{cext}
0 \to \CO_C(-3) \xrightarrow{\ P_1 + P_2\ } \CO_C(-1) \to \CO_{P_1} \oplus \CO_{P_2} \to 0
\end{equation}
corresponding to the sum of extension $\eta_1$ and $\eta_2$, and apply the functor $\Hom(i_*\CO_E(1,0),-)$ to it.
We will get an exact sequence
\begin{multline*}
\dots \to 
\Hom(i_*\CO_E(1,0),j_*\CO_C(-1)) \to \\ \to
\Hom(i_*\CO_E(1,0),\CO_{P_1}) \oplus
\Hom(i_*\CO_E(1,0),\CO_{P_2}) \to \\ \to
\Ext^1(i_*\CO_E(1,0),j_*\CO_C(-3)) \to \dots
\end{multline*}
The first term is zero (this is proved analogously to~\eqref{e23}).
Consequently, the second map is an embedding.
Clearly, it takes the basis $(\rho_1,\rho_2)$ of the second term to $j_*(\eta_1) \circ \rho_1$
and $j_*(\eta_2) \circ \rho_2$ respectively, which thus span the second space. Clearly, these
elements coincide with $\beta_1$ and $\beta_2$.

Now we can check that the products of $\alpha_1\otimes\beta_2$ and of $\alpha_2\otimes\beta_1$ are zero.
Indeed,
\begin{equation*}
m(\alpha_1\otimes\beta_2) = 
\beta_2\circ\alpha_1 =
j_*(\eta_2) \circ \rho_2 \circ \alpha_1
\end{equation*}
and already the composition
\begin{equation*}
\rho_2\circ\alpha_1 : \CO_X \xrightarrow{\ \alpha_1\ } \CO_E(1,0) \xrightarrow{\ \rho_2\ } \CO_{P_2}
\end{equation*}
is zero since $\alpha_1$ vanishes at point $P_2$. The same argument applies to $\alpha_2\otimes\beta_1$.

So, to finish the proof of the Proposition it remains to check that the products of $\alpha_1\otimes\beta_1$ 
and $\alpha_2\otimes\beta_2$ are linearly independent in the space $\Ext^1(\CO_X,j_*\CO_C(-3))$. For this we note that 
\begin{equation*}
m(\alpha_i\otimes\beta_i) = 
\beta_i\circ\alpha_i =
j_*(\eta_i) \circ \rho_i \circ \alpha_i
\end{equation*}
and that the composition
\begin{equation*}
\rho_i\circ\alpha_i : \CO_X \xrightarrow{\ \alpha_i\ } \CO_E(1,0) \xrightarrow{\ \rho_i\ } \CO_{P_i}
\end{equation*}
is equal to the canonical evaluation map $e_i:\CO_X \to \CO_{P_i}$. Finally, applying the functor $\Hom(\CO_X,-)$ to sequence~\eqref{cext}, we obtain
an exact sequence
\begin{equation*}
\dots \to \Hom(\CO_X,j_*\CO_C(-1)) \to \Hom(\CO_X,\CO_{P_1}) \oplus \Hom(\CO_X,\CO_{P_2}) \to \Ext^1(\CO_X,j_*\CO_C(-3)) \to \dots
\end{equation*}
Again, its first term is zero since the line bundle $\CO_C(-1)$ on the rational curve $C$ is acyclic,
hence the second map is an embedding. This means that the images $j_*(\eta_i)\circ e_i$ of the 
canonical evaluation maps $e_i$ are linearly independent in $\Ext^1(\CO_X,j_*\CO_C(-3))$. This precisely
means that $m(\alpha_i\otimes\beta_i)$ are linearly independent.
\end{proof}

\end{document}